\documentclass[]{interact}

\usepackage{epstopdf}
\usepackage[caption=false]{subfig}
\usepackage{color}

\usepackage[numbers,sort&compress]{natbib}
\bibpunct[, ]{[}{]}{,}{n}{,}{,}

\usepackage{ mathrsfs }

\theoremstyle{plain}
\newtheorem{theorem}{Theorem}[section]
\newtheorem{lemma}[theorem]{Lemma}

\newtheorem{proposition}[theorem]{Proposition}

\theoremstyle{definition}

\theoremstyle{remark}

\newcommand{\qbinomial}{\genfrac{[}{]}{0pt}{}}

\begin{document}


\title{Second--order difference equation for Sobolev--type orthogonal polynomials. Part I: theoretical results}

\author{
\name{Galina Filipuk\textsuperscript{a}, Juan F. Ma\~{n}as--Ma\~{n}as\textsuperscript{b}\thanks{CONTACT Juan F. Ma\~{n}as--Ma\~{n}as. Email: jmm939@ual.es} and Juan J. Moreno--Balc\'{a}zar\textsuperscript{b,c}}
\affil{\textsuperscript{a}Institute of Mathematics, Faculty of Mathematics, Informatics and Mechanics, University of Warsaw, Poland\\ \textsuperscript{b}Departamento de Matem\'{a}ticas, Universidad de Almer\'{\i}a, Spain \\ \textsuperscript{c}Instituto Carlos I de F\'{\i}sica Te\'{o}rica y Computacional, Universidad de Granada, Spain}
}

\maketitle

\begin{abstract}
We consider a general Sobolev--type inner product involving the Hahn difference operator, so this includes the well--known difference operators $\mathscr{D}_{q}$ and $\Delta$ and, as a limit case, the derivative operator. The objective is to construct the ladder operators for the corresponding nonstandard orthogonal polynomials  and   deduce the second--order differential--difference equation satisfied by these polynomials. Moreover, we will show that all the functions involved in these constructions can be computed explicitly.
\end{abstract}


\section{Introduction}\label{introducction}

The literature about Sobolev orthogonal polynomials is very extensive, dating back to the sixties of the last century although the seminal paper on this topic is due to Lewis in 1947. End of the eighties and the nineties saw a burst of papers about Sobolev--type orthogonal polynomials. A general and historic vision until 2015 is given in the survey \cite{marxu} where the authors rename these nonstandard orthogonal polynomials as Sobolev orthogonal polynomials of the second type, although  the most frequently used names are Sobolev--type orthogonal polynomials and discrete Sobolev orthogonal polynomials. In some of those papers the  Sobolev--type inner product
$$
(f,g)_S=\int_{\mathbb{R}} f(x)g(x)\varrho(x)dx+ M f^{(j)}(c) g^{(j)}(c), \quad M>0, \quad j\in \mathbb{N},
$$
was considered, where $\varrho(x)$ is a  weight function with support on the real line and $c$ is  located on the real axis. The authors of those papers studied extensively algebraic, differential and asymptotic properties of the corresponding Sobolev orthogonal polynomials (SOP), including generalizations of the inner product or even  the varying case,  where the constant $M$ is changed by a general sequence depending on $n$ (see \cite{SGV16} for the varying case).  Notice  if we allow $j=0,$ then the inner product will be standard and, therefore, we have all the advantages of the standard polynomials such as the three--term recurrence relation, Christoffel--Darboux formula, etc. Thus, to use the word Sobolev properly we should take $j$ as a positive integer, i.e. $j\in \mathbb{N}$.

Recently, there has been a growing interest in considering nonstandard inner products with noncontinuous weights, i.e. weights supported on  some lattices. According to the weight, the differential operator is substituted  either by the difference operator $\Delta$ or by the $q$--difference operator $\mathscr{D}_{q}.$  This leads to $\Delta$--SOP or $q$--SOP.  In many papers (e.g. \cite{AMRR, acmp1, area-godoy-Marcellan-2002,area-godoy-Marcellan-2003,area-godoy-et-2000-2,Bavinck-Koekoek-1998,cp,cjmp,Roberto-Anier-2018,Koekoek-1990-CJM,
Koekoek-1992-JAT,JJMB-2015,Sha-Gad-2016}) these SOP were studied  by considering  particular weights or by taking  one of these two difference operators.  However, in other papers (e.g. \cite{acmp2,Renato-et-2014-JCAM,costas-juanjo-2011}) these three operators were treated in a unified way via the Hahn difference operator.

Ladder operators for orthogonal polynomials have been extensively studied in the literature, see  \cite{Chen-Griffin-2002,Chen-Ismail-1997, Chen-Ismail-2008, Chen-Its-2010,Chen-Pruessner-2005,Ismail-2005,walter-2017} among other references. One reason for this, but not the only one, is that they are a nice and useful tool to construct differential (or difference) equations satisfied by orthogonal polynomials.

Thus, the main goal of this paper is the study of the ladder operators for a wide class of  Sobolev--type orthogonal polynomials using the Hahn difference operator as a powerful tool and, as a consequence, to provide the second--order differential--difference equation satisfied by these nonstandard polynomials. In addition,   we generalize some results of other contributions to this topic where the authors considered particular cases, e.g., \cite{{Arceo-H-M-2016},{Fil-JFMM-JuanJo-2018},{Garza-et-2019},{Huertas-Soria-2019}}.

In this paper, we consider the Sobolev--type inner product

\begin{equation}\label{inner-product}
(f,g)_S=\int_{\mathbb{R}} f(x)g(x)\varrho(x)dx+M\mathscr{D}_{q,\omega}^{(j)}f(c)\mathscr{D}_{q,\omega}^{(j)}g(c),
\end{equation}
where $\varrho(x)$ is a weight  function with support on the real line, $c$ is  located on the real axis, $M>0$, $j$ is a non--negative integer and $\mathscr{D}_{q,\omega}$ is the operator introduced by Hahn in \cite[Eq. (1.3)]{Hahn-1949} (see  also  \cite[Eq. (2.1.1)]{Koekoek-book-hyper})  defined by
\begin{equation}\label{hahn-operator}
\mathscr{D}_{q,\omega}f(x)=\left\{
                      \begin{array}{ll}
                       \displaystyle \frac{f(qx+\omega)-f(x)}{(q-1)x+\omega}, & \hbox{if $x\neq\omega_0$,} \\
                        f'(\omega_0), & \hbox{if $x=\omega_0$,}
                      \end{array}
                    \right.
\end{equation}
where $0<q<1$, $\omega\geq0$, $\omega_0=\frac{\omega}{1-q},$ and following the notation given in \cite{Hamza-et-2012}, we define
$$\mathscr{D}_{q,\omega}^{(0)}f:=f \quad \mathrm{and} \quad \mathscr{D}_{q,\omega}^{(j)}f:=\mathscr{D}_{q,\omega}\mathscr{D}_{q,\omega}^{(j-1)}f,\quad j\in \mathbb{N}.$$

Thus, $\mathscr{D}_{q,\omega}^{(j)}f(c)$ in (\ref{inner-product}) means
\begin{equation}\label{notation}
\mathscr{D}_{q,\omega}^{(j)}f(c)=\left.\mathscr{D}_{q,\omega}^{(j)}f(x)\right|_{x=c}.
\end{equation}
To simplify the notation we will write $\mathscr{D}_{q,\omega}^{(j)}f(c)$ instead of $\left.\mathscr{D}_{q,\omega}^{(j)}f(x)\right|_{x=c}$ when no confusion arises. Notice when $\varrho(x)$ is a discrete weight, then the integral in (\ref{inner-product}) must be considered adequately, so that according to the case, it could be the Jackson--N\"{o}rlund integral (see \cite{Hamza-et-2012}) or  the Jackson integral.

The class of operators given by (\ref{hahn-operator}) includes
the $q$-difference operator $\mathscr{D}_{q}$ introduced by Jackson  when $\omega = 0$, the  forward difference operator $\Delta$  when  $q = 1$ and $\omega = 1,$
and  the derivative operator $\frac{d}{dx}$ as the limiting case when $\omega = 0$ and $q \to 1$.

 When $q=1$ the difference operator $\mathscr{D}_{\omega}$ acts on general lattices, with $\Delta$ as a particular case. However, we  will only give the explicit formulas in  Appendices A, B, C for the three operators:   $\mathscr{D}_{q},$ $\Delta, $ and $\frac{d}{dx}.$  We   will not include an appendix for the operator $\mathscr{D}_{\omega}$,  because the above mentioned three operators are most commonly used  in the literature in the framework of the Sobolev orthogonality (see the survey \cite{marxu}), but the results hold for any $\omega\ge 0.$

Let $\{Q_n\}_{n\geq0}$ be the infinite sequence of monic orthogonal polynomials with respect to the inner product (\ref{inner-product}).
Also, let $\{P_n\}_{n\geq 0}$ be the infinite sequence of monic polynomials  orthogonal  with respect to the inner product

\begin{equation}\label{standart-product}
(f,g)_{\varrho}=\int_{\mathbb{R}} f(x)g(x)\varrho(x)dx.
\end{equation}
Thus,

\begin{equation*}\label{h_n}
(P_n,P_k)_{\varrho}=\int_{\mathbb{R}} P_n(x)P_k(x)\varrho(x)dx=h_n\delta_{n,k},\qquad n,k\in\mathbb{N}\cup\{0\},
\end{equation*}
where $\delta_{n,k}$ denotes the  Kronecker delta and $h_n$ is the  square
of the norm of these polynomials. Clearly, we assume that the two previous sequences of polynomials are infinite.

Since the inner product (\ref{standart-product})  is standard, i.e., the property $(xf,g)_{\varrho}=(f,xg)_{\varrho}$ holds, then the sequence of monic orthogonal polynomials $\{P_n\}_{n\geq 0}$ satisfies a three-term recurrence relation of the following form:

\begin{equation}\label{ttrr}
xP_n(x)=P_{n+1}(x)+\alpha_nP_n(x)+\beta_nP_{n-1}(x),\qquad n\geq0,
\end{equation}
with  initial conditions $P_{-1}(x)=0$ and $P_0(x)=1.$ In addition, $\beta_n=h_n/h_{n-1}$ for  $n\geq1.$
When the weight  $\varrho$ is symmetric,  $\alpha_n=0$ (see, for example, \cite{Chihara-1978}).

As we have mentioned previously, our objective is to obtain ladder operators and a linear second--order differential--difference equation for  sequence of  monic  Sobolev orthogonal polynomials $\{Q_n\}_{n\geq0}$.

We assume the sequence of  monic orthogonal polynomials $\{P_n\}_{n\geq 0}$  with respect to (\ref{standart-product}) satisfies the following relation:

\begin{equation}\label{Low-rela-Pn}
A(x)\mathscr{D}_{q,\omega}P_n(x)=B_n(x)P_n(x)+C_n(x)P_{n-1}(x), \qquad n\geq1,
\end{equation}
where $B_n(x)$ and $C_n(x)$ are certain rational functions and $A(x)$ is a polynomial. Notice relation (\ref{Low-rela-Pn}) is very general. It holds for general standard orthogonal polynomials with respect to inner products involving any of the three operators  $\mathscr{D}_{q},$ $\Delta, $ and $\frac{d}{dx}$  (see \cite{Ismail-2005}).

Then, we define the lowering operator from (\ref{Low-rela-Pn}) as
\begin{equation*}\label{Low-operator-Pn}
\Psi_n:=A(x)\mathscr{D}_{q,\omega}-B_n(x), \qquad n\geq1,
\end{equation*}
hence,
\begin{equation*}\label{Low-operator-Pn-2}\Psi_nP_n(x)=C_n(x)P_{n-1}(x), \qquad n\geq1.\end{equation*}

The raising operator is defined as follows:
\begin{equation*}\label{Rai-operator-Pn}
\hat{\Psi}_n:=B_{n-1}(x)+\frac{C_{n-1}(x)(x-\alpha_{n-1})}{\beta_{n-1}}-A(x)\mathscr{D}_{q,\omega}, \qquad n\geq2,
\end{equation*}
hence, using (\ref{ttrr}), we deduce
\begin{equation*}\label{proof-1}
\hat{\Psi}_nP_{n-1}(x)=\frac{C_{n-1}(x)}{\beta_{n-1}}P_{n}(x), \qquad n\geq2.
\end{equation*}

Once obtained the suitable properties of these ladder operators, we can find the second--order differential--difference equation for the polynomials $Q_n(x).$

Our approach is based on obtaining ladder operators for this wide class of Sobolev--type orthogonal polynomials and then, as a consequence, to deduce the second--order differential--difference equation satisfied by these polynomials. In this sense, we follow the steps by Ismail in his book \cite{Ismail-2005} for standard continuous orthogonal polynomials, firstly obtaining the ladder operators and secondly the corresponding second--order differential equation.  Thus, we have unified this approach by considering a very general operator $\mathscr{D}_{q,\omega}$ which includes, as far as we know, all the operators used in the literature in the framework of Sobolev--type orthogonality. Furthermore, Appendices A, B, and C allow us to obtain efficiently, via a symbolic computer program,  all the functions involved both in the ladder operators and in the second--order differential--difference equation.

This paper is structured as follows. In Section 2 we introduce some basic notation and provide a useful connection formula for the SOP, $Q_n,$ in terms of the standard orthogonal polynomials, $P_n.$ Section 3 is devoted to  finding  several relationships between  these two families of orthogonal polynomials. The use of the Hahn difference operator and its properties will be essential to deduce these relationships which will be the key to obtain the main results of this paper.  Thus, in Section 4 we obtain the expressions for the ladder operators and the second--order differential--difference equation satisfied by the SOP, which constitute the goals of this work as  we have mentioned previously. Finally, we provide one appendix for each one of the three operators: $\mathscr{D}_{q},$ $\Delta, $ and $\frac{d}{dx}.$ In these appendices we give explicitly the expressions of the functions involved in the construction of the ladder operators as well as in the coefficients of the second--order differential--difference equations.

\section{Connection formula}

In this section we provide a connection formula  for the  orthogonal polynomials with respect to (\ref{inner-product}), $\{Q_n\}_{n\geq 0}$,  in terms of the orthogonal polynomials with respect to (\ref{standart-product}), $\{P_n\}_{n\geq 0}$. In fact,  this formula follows from a well--known standard technique which we have modified slightly using the Hahn operator.

 We introduce some notation that will be used throughout the paper.  We define the kernel polynomials in the usual way by

\begin{equation}\label{kernel}
K_n(x,y):=\sum_{i=0}^{n}\frac{P_i(x)P_i(y)}{h_i}.
\end{equation}
The kernel polynomials (\ref{kernel}) satisfy the  Christoffel--Darboux formula (see, for example, \cite{Chihara-1978} or \cite{sz})
\begin{equation}\label{Chri-Dar-for}
K_n(x,y)=\frac{1}{h_n}\frac{P_{n+1}(x)P_n(y)-P_n(x)P_{n+1}(y)}{x-y}.
\end{equation}

As usual, we denote
\begin{equation}\label{derivatives-kernel}
K_n^{(k,\ell)}(x,y):=\sum_{i=0}^{n}\frac{\mathscr{D}_{q,\omega}^{(k)}P_i(x)\ \mathscr{D}_{q,\omega}^{(\ell)}P_i(y)}{h_i}, \qquad k,\ell\in\mathbb{N}\cup\{0\}.
\end{equation}

In order to obtain the connection formula, we use a similar technique to the one used in  \cite[Lemma 2]{SGV16} or \cite[Section 2]{maRon}.
Since the sequence of  orthogonal polynomials $\{P_n\}_{n\geq 0}$ is a basis of the linear space $\mathbb{P}_n[x]$ of polynomials with real coefficients of degree at most $n,$  we write
\begin{equation}\label{q-relation-1}
Q_n(x)=\sum_{k=0}^n a_{n,k}P_k(x).
\end{equation}
The coefficient $a_{n,n}=1$ since $Q_n(x)$ and $P_n(x)$ are monic polynomials. For $0\leq k\leq n-1$, using the orthogonality of $Q_n(x)$ and $P_n(x)$, we immediately obtain
\begin{equation}\label{ank}
a_{n,k}=-\frac{M\mathscr{D}_{q,\omega}^{(j)}Q_n(c)\mathscr{D}_{q,\omega}^{(j)}P_k(c)}{h_k}.
\end{equation}
So, considering  (\ref{kernel}), (\ref{q-relation-1}) and (\ref{ank}),

\begin{equation}
Q_n(x)=P_n(x)-M\mathscr{D}_{q,\omega}^{(j)}Q_n(c)K_{n-1}^{(0,j)}(x,c).\label{Qn-connection-formula}
\end{equation}

To compute the value of $\mathscr{D}_{q,\omega}^{(j)}Q_n(c)$, we apply  the Hahn operator $\mathscr{D}_{q,\omega}$ in (\ref{Qn-connection-formula}) getting
\begin{equation*}\label{value-of-D(j)Qn(c)}
\mathscr{D}_{q,\omega}^{(j)}Q_n(c)=\frac{\mathscr{D}_{q,\omega}^{(j)}P_n(c)}{1+MK_{n-1}^{(j,j)}(c,c)}.
\end{equation*}
Therefore, the connection formula between both families of orthogonal polynomials is given by
\begin{equation}\label{qn-relation-first}
Q_n(x)=P_n(x)-\frac{M\mathscr{D}_{q,\omega}^{(j)}P_n(c)}{1+MK_{n-1}^{(j,j)}(c,c)}K_{n-1}^{(0,j)}(x,c).
\end{equation}

\section{Hahn difference operator and technical relations between families of orthogonal polynomials $\boldsymbol{Q_n(x)}$ and $\boldsymbol{P_n(x)}$}
In this section we will obtain different relations between $P_n(x)$ and $Q_n(x)$ which will be used to find the ladder operators for the polynomials $Q_n(x)$.
First, we need some properties of the Hahn difference operator  collected in the next statement.

\begin{proposition}
The Hahn difference operator has the following  properties:
 \begin{enumerate}
\item Linearity:
\begin{equation*}\label{h-linear}
\mathscr{D}_{q,\omega}\left(f+g\right)(x)=\mathscr{D}_{q,\omega} f(x)+\mathscr{D}_{q,\omega} g(x),
\end{equation*}
\begin{equation*}\label{h-linear1}
\mathscr{D}_{q,\omega}\left(af\right)(x)=a\mathscr{D}_{q,\omega} f(x), \quad a \in \mathbb{R}.
\end{equation*}
\item Product rule:
\begin{equation}\label{h-rule-product}
\mathscr{D}_{q,\omega}(fg)(x)=g(x)\mathscr{D}_{q,\omega} f(x)+f(qx+\omega)\mathscr{D}_{q,\omega} g(x).
\end{equation}
\item Quotient rule:
\begin{equation}\label{h-quotient-rule}
\mathscr{D}_{q,\omega}\left(\frac{f}{g}\right)(x)=\frac{g(x)\mathscr{D}_{q,\omega}f(x)-f(x)\mathscr{D}_{q,\omega}
g(x)}{g(x)g(qx+\omega)}.
\end{equation}
\item  Leibniz formula:
\begin{equation}\label{h-Leibniz}
\mathscr{D}_{q,\omega}^{(n)}(fg)(x)=\sum_{k=0}^{n}\qbinomial{n}{k}_{q}q^{k(k-n)}\mathscr{D}_{q,\omega}^{(k)} g\left(x\right)\mathscr{D}_{q,\omega}^{(n-k)} f\left(q^k x+\omega[k]_q\right),\qquad n\geq0,
\end{equation}
where the $q$-binomial coefficient is defined by (see, \cite[Eq. (1.9.4)]{Koekoek-book-hyper})
\begin{equation}\label{q-binomial}
\qbinomial{n}{k}_{q}:=\frac{(q;q)_n}{(q;q)_k (q;q)_{n-k}},
\end{equation}
with $(a;q)_k$ being a $q$-analogue of the Pochhammer symbol $(a)_k$ defined by (see, \cite[Eq. (1.8.3)]{Koekoek-book-hyper})
\begin{equation}\label{q-Pochhammer}
(a ; q)_{0}:=1 \quad \text { and } \quad(a ; q)_{k}:=\prod_{i=1}^{k}\left(1-a q^{i-1}\right), \qquad k\in \mathbb{N},
\end{equation}
and for $q\neq0$ and $q\neq1$    the basic $q$-number is given by (see, \cite[Eq. (1.8.1)]{Koekoek-book-hyper})
\begin{equation}\label{q-number}
[\alpha]_q:=\frac{1-q^{\alpha}}{1-q}, \qquad \alpha\in \mathbb{R}.
\end{equation}
\item Let $\alpha,\gamma\in \mathbb{R}$. We take $f(x)=(\beta-\gamma x)^{-1}$, then
\begin{equation}\label{h-derivate-qoutient}
\mathscr{D}_{q,\omega}^{(n)}f(x)=\frac{\gamma^n (q;q)_n}{(1-q)^{n}\prod_{i=0}^n(\beta-\gamma (q^{i}x+\omega[i]_q))}.
\end{equation}
\end{enumerate}
\end{proposition}

\begin{proof}
The linearity follows directly from the definition. The product and quotient rules can be found, for example,  in \cite[Eq. (16)-(17)]{Hamza-et-2012}.

To establish (\ref{h-Leibniz}) we need the formula
\begin{equation*}
\mathscr{D}_{q,\omega}^{(n)}(fg)(x)=\sum_{k=0}^{n}\qbinomial{n}{k}_{q}\mathscr{D}_{q,\omega}^{(k)}  g\left(x\right)\left(\mathscr{D}_{q,\omega}^{(n-k)}f\right) \left(q^k x+\omega[k]_q\right),
\end{equation*}
given in \cite[Theorem 3.1]{Hamza-et-2012}. Then, it is enough to apply \cite[Eq. (26)]{Hamza-et-2012} repeatedly to obtain the result:
\begin{align*}
\mathscr{D}_{q,\omega}^{(n)}(fg)(x)
&=\sum_{k=0}^{n}\qbinomial{n}{k}_{q}\mathscr{D}_{q,\omega}^{(k)} g\left(x\right)\left(\mathscr{D}_{q,\omega}^{(n-k)}f\right) \left(q^k x+\omega[k]_q\right)\\
&=\sum_{k=0}^{n}\qbinomial{n}{k}_{q}\mathscr{D}_{q,\omega}^{(k)} g\left(x\right)q^{-k}\mathscr{D}_{q,\omega}\left(\mathscr{D}_{q,\omega}^{(n-k-1)}f\right) \left(q^k x+\omega[k]_q\right)\\
&=\sum_{k=0}^{n}\qbinomial{n}{k}_{q}\mathscr{D}_{q,\omega}^{(k)} g\left(x\right)q^{-2k}\mathscr{D}_{q,\omega}^{(2)}\left(\mathscr{D}_{q,\omega}^{(n-k-2)}f\right) \left(q^k x+\omega[k]_q\right) \\
&=\cdots\\
&=\sum_{k=0}^{n}\qbinomial{n}{k}_{q}q^{k(k-n)}\mathscr{D}_{q,\omega}^{(k)}g\left(x\right)\mathscr{D}_{q,\omega}^{(n-k)}f(q^k x+\omega[k]_q).
\end{align*}

Finally, to prove (\ref{h-derivate-qoutient}), we use mathematical induction on $n$.
For $f(x)=(\beta-\gamma x)^{-1}$ and $n=1$, we have
\begin{eqnarray*}
\mathscr{D}_{q,\omega}f(x)
=\frac{(\beta-\gamma (qx+\omega))^{-1}-(\beta-\gamma x)^{-1}}{(qx+\omega)-x}
=\frac{\gamma}{(\beta-\gamma x)(\beta-\gamma (qx+\omega))}.
\end{eqnarray*}
Since $(q;q)_1=1-q$, (\ref{h-derivate-qoutient}) holds for $n=1$.

Assuming the identity (\ref{h-derivate-qoutient}) holds for some $n\ge 2$,  we will prove it for $n+1$. To do this, the following relations will be useful:
\begin{equation}\label{q-number-Gasper-property}
\prod_{k=1}^n[k]_q=\prod_{k=0}^{n-1}[1+k]_q=(1-q)^{-n}(q;q)_n,
\end{equation}
which can be deduced from  \cite[Eq. (1.2.45)]{Gasper-Rahman-2004} and \cite[Eq. (1.2.47)]{Gasper-Rahman-2004}).
Then, by (\ref{q-number-Gasper-property}), establishing  (\ref{h-derivate-qoutient}) is equivalent to proving
\begin{equation*}
\mathscr{D}_{q,\omega}^{(n)}f(x)=\frac{\gamma^n(1-q)^{-n}(q;q)_n}{\prod_{i=0}^n(\beta-\gamma (q^{i}x+\omega[i]_q))}.
\end{equation*}

We denote $$f_n(x):=\frac{\gamma^n(1-q)^{-n}(q;q)_n}{\prod_{i=0}^n(\beta-\gamma (q^{i}x+\omega[i]_q))}.$$
Then, using the induction hypothesis and after some calculations, we find
\begin{align*}
&\mathscr{D}_{q,\omega}^{(n+1)}f(x)
=\mathscr{D}_{q,\omega}\mathscr{D}_{q,\omega}^{(n)}f(x)
=\mathscr{D}_{q,\omega}f_n(x)
=\frac{f_n(qx+\omega)-f_n(x)}{(q-1)x+\omega}\\
=&\gamma^n(1-q)^{-n}(q;q)_n\ \frac{\left(\prod_{i=0}^n(\beta-\gamma (q^{i}(qx+\omega)+\omega[i]_q))\right)^{-1}-\left(\prod_{i=0}^n(\beta-\gamma (q^{i}x+\omega[i]_q))\right)^{-1}}{(q-1)x+\omega}\\
=&\frac{\gamma^{n+1}(1-q)^{-n-1}(q;q)_{n+1}}{\prod_{i=0}^{n+1}(\beta-\gamma (q^{i}x+\omega[i]_q))},
\end{align*}
and this completes the proof.
\end{proof}

\bigskip

The following five lemmas are essential to obtain our main results in Section \ref{sect-lo}.

\begin{lemma}\label{h-lemma-1}
Let  $\{Q_n(x)\}_{n\geq0}$ and $\{P_n(x)\}_{n\geq0}$ be the sequences of monic  orthogonal polynomials  with respect to (\ref{inner-product}) and (\ref{standart-product}), respectively. Then,
\begin{equation}\label{relation-1}
r_c(x)Q_n(x)=f_{1,c,n}(x)P_n(x)+g_{1,c,n}(x)P_{n-1}(x),\qquad n\geq1,
\end{equation}
where
\begin{align}\label{r(x)}
r_c(x)&= \prod_{k=0}^{j}(x-q^kc-\omega[k]_q),\\ \label{f1}
f_{1,c,n}(x)&=r_c(x)-\frac{M\rho_{n,j,c}}{h_{n-1}}
\left(\sum_{i=0}^j\frac{(q;q)_j\mathscr{D}_{q,\omega}^{(i)}P_{n-1}\left(c\right)\prod_{k=0}^{i-1}(x-q^kc-\omega[k]_q)}
{(q;q)_{i}(1-q)^{j-i}}\right),\\  \label{g1}
g_{1,c,n}(x)&=\frac{M\rho_{n,j,c}}{h_{n-1}}\left(\sum_{i=0}^j\frac{(q;q)_j\mathscr{D}_{q,\omega}^{(i)}
P_{n}\left(c\right)\prod_{k=0}^{i-1}(x-q^kc-\omega[k]_q)} {(q;q)_{i}(1-q)^{j-i}}\right),
\end{align}
with $\rho_{n,j,c}:=\mathscr{D}_{q,\omega}^{(j)}Q_n(c)=\displaystyle{\frac{\mathscr{D}_{q,\omega}^{(j)}P_n(c)}{1+MK_{n-1}^{(j,j)}(c,c)}}.$
\end{lemma}
\begin{proof}
Using (\ref{Chri-Dar-for}), (\ref{derivatives-kernel}), and applying  (\ref{h-Leibniz})-(\ref{h-derivate-qoutient}), we deduce
\begin{align*}
&K_{n-1}^{(0, j)}(x, y)\\
=&\frac{1}{h_{n-1}}\left(P_n(x)\mathscr{D}_{q,\omega}^{(j)}
\frac{P_{n-1}(y)}{x-y}-P_{n-1}(x)\mathscr{D}_{q,\omega}^{(j)}\frac{P_{n}(y)}{x-y}\right)\\
=&\frac{1}{h_{n-1}}\left(P_n(x)\sum_{i=0}^j\qbinomial{j}{i}_qq^{i(i-j)}\mathscr{D}_{q,\omega}^{(i)}P_{n-1}\left(y\right)
\mathscr{D}_{q,\omega}^{(j-i)}(x-(q^{i}y+\omega[i]_q))^{-1}\right.\\
&\qquad\left.-P_{n-1}(x)\sum_{i=0}^j\qbinomial{j}{i}_qq^{i(i-j)}\mathscr{D}_{q,\omega}^{(i)}P_{n}\left(y\right)
\mathscr{D}_{q,\omega}^{(j-i)}(x-(q^{i}y+\omega[i]_q))^{-1}\right)\\
=&\frac{1}{h_{n-1}}\left(P_n(x)\sum_{i=0}^j\frac{(q;q)_j}{(q;q)_i}
\mathscr{D}_{q,\omega}^{(i)}P_{n-1}\left(y\right)\frac{1}{(1-q)^{j-i}\prod_{k=0}^{j-i}(x-q^{k+i}y-\omega[k+i]_q)}\right.\\
&\qquad\left.-P_{n-1}(x)\sum_{i=0}^j\frac{(q;q)_j}{(q;q)_i}
\mathscr{D}_{q,\omega}^{(i)}P_{n}\left(y\right)\frac{1}{(1-q)^{j-i}\prod_{k=0}^{j-i}(x-q^{k+i}y-\omega[k+i]_q)}\right).
\end{align*}
So, evaluating the variable $y$ at  point $c$,  and
 substituting it into relation (\ref{qn-relation-first}) we obtain

\begin{align*}
Q_n(x)
=&P_n(x)-M\rho_{n,j,c}K_{n-1}^{(0,j)}(x,c)\\
=&P_n(x)\left(1-\frac{M\rho_{n,j,c}}{h_{n-1}}\left(\sum_{i=0}^j\frac{(q;q)_j\mathscr{D}_{q,\omega}^{(i)}P_{n-1}\left(c\right)}
{(q;q)_{i}(1-q)^{j-i}\prod_{k=0}^{j-i}(x-q^{k+i}c-\omega[k+i]_q)}\right)\right)\\
+&P_{n-1}(x)\frac{M\rho_{n,j,c}}{h_{n-1}}\left(\sum_{i=0}^j\frac{(q;q)_j\mathscr{D}_{q,\omega}^{(i)}P_{n}\left(c\right)}
{(q;q)_{i}(1-q)^{j-i}\prod_{k=0}^{j-i}(x-q^{k+i}c-\omega[k+i]_q)}\right).
\end{align*}
Finally, multiplying by $r_c(x)=\prod_{k=0}^{j}(x-q^kc-\omega[k]_q)$,
\begin{align*}
&Q_n(x)\prod_{k=0}^{j}(x-q^kc-\omega[k]_q)\\
=&P_n(x)\left(\prod_{k=0}^{j}(x-q^kc-\omega[k]_q)-\frac{M\rho_{n,j,c}}{h_{n-1}}
\left(\sum_{i=0}^j\frac{(q;q)_j\mathscr{D}_{q,\omega}^{(i)}P_{n-1}\left(c\right)\prod_{k=0}^{j}(x-q^kc-\omega[k]_q)}
{(q;q)_{i}(1-q)^{j-i}\prod_{k=0}^{j-i}(x-q^{k+i}c-\omega[k+i]_q)}\right)\right)\\
+&P_{n-1}(x)\frac{M\rho_{n,j,c}}{h_{n-1}}\left(\sum_{i=0}^j\frac{(q;q)_j\mathscr{D}_{q,\omega}^{(i)}
P_{n}\left(c\right)\prod_{k=0}^{j}(x-q^kc-\omega[k]_q)}
{(q;q)_{i}(1-q)^{j-i}\prod_{k=0}^{j-i}(x-q^{k+i}c-\omega[k+i]_q)}\right),
\end{align*}
and simplifying we obtain the desired result.
\end{proof}

\begin{lemma}\label{h-lemma-2}
Let   $\{Q_n(x)\}_{n\geq0}$ and $\{P_n(x)\}_{n\geq0}$ be the sequences of monic orthogonal polynomials  with respect to (\ref{inner-product}) and (\ref{standart-product}), respectively. Then,
\begin{equation}\label{relation-2}
r_c(x)\mathscr{D}_{q,\omega}Q_n(x)=f_{2,c,n}(x)P_n(x)+g_{2,c,n}(x)P_{n-1}(x),\quad n\geq1,
\end{equation}
where
\begin{align}\nonumber
&f_{2,c,n}(x)=\frac{x-q^{j}c-\omega[j]_q}{q^{j+1}\left(x-q^{-1}c-\omega[-1]_q\right)}\left(\mathscr{D}_{q,\omega}
f_{1,c,n}(x)+f_{1,c,n}(qx+\omega)\frac{B_n(x)}{A(x)}\right.\\ \label{f2}
&\quad\left.-g_{1,c,n}(qx+\omega)\frac{C_{n-1}(x)}{\beta_{n-1}A(x)}-\frac{[j+1]_q}{\left(x-q^{j}c-\omega[j]_q\right)}f_{1,c,n}(x)\right),\\ \nonumber
&g_{2,c,n}(x)=\frac{x-q^{j}c-\omega[j]_q}{q^{j+1}\left(x-q^{-1}c-\omega[-1]_q\right)}
\left(\mathscr{D}_{q,\omega}\ g_{1,c,n}(x)+f_{1,c,n}(qx+\omega)\frac{C_n(x)}{A(x)}\right.\\ \label{g2}
&\quad\left.+g_{1,c,n}(qx+\omega)\left(\frac{B_{n-1}(x)}{A(x)}+\frac{(x-\alpha_{n-1})C_{n-1}(x)}{A(x)\beta_{n-1}}\right)
-\frac{[j+1]_q}{\left(x-q^{j}c-\omega[j]_q\right)}g_{1,c,n}(x)\right),
\end{align}
with $r_c(x)$, $f_{1,c,n}(x)$ and $g_{1,c,n}(x)$  given  by (\ref{r(x)})--(\ref{g1}).
\end{lemma}

\begin{proof}
Applying the operator $\mathscr{D}_{q,\omega}$  to (\ref{relation-1}) in Lemma \ref{h-lemma-1}, we get
\begin{eqnarray}\nonumber
&&Q_n(x)\mathscr{D}_{q,\omega} r_c(x)
+r_c(qx+\omega)\mathscr{D}_{q,\omega}Q_n(x)\\ \nonumber
&=&P_n(x)\mathscr{D}_{q,\omega}f_{1,c,n}(x)+f_{1,c,n}(qx+\omega)\mathscr{D}_{q,\omega}P_n(x)\\ \label{need-lemma-2}
&+&P_{n-1}(x)\mathscr{D}_{q,\omega}g_{1,c,n}(x)+g_{1,c,n}(qx+\omega)\mathscr{D}_{q,\omega}P_{n-1}(x).
\end{eqnarray}

We are going to  compute  some terms in the previous relation starting with $\mathscr{D}_{q,\omega}r_c(x).$
\begin{align}\nonumber
\mathscr{D}_{q,\omega}r_c(x)&=\frac{r_c(qx+\omega)-r_c(x)}{(q-1)x+\omega}\\ \nonumber
&=\frac{q^{j+1}\prod_{k=0}^{j}(x-q^{k-1}c-\omega [k-1]_q)-\prod_{k=0}^{j}(x-q^kc-\omega[k]_q)}{(q-1)x+\omega}\\ \label{eqA}
&=[j+1]_q\prod_{k=0}^{j-1}(x-q^kc-\omega[k]_q)=\frac{[j+1]_q}{x-q^{j}c-\omega[j]_q}r_c(x).
\end{align}
Next, we express  $r_c(qx+\omega)=\prod_{k=0}^{j}(qx+\omega-q^{k}c-\omega[k]_q)$ in terms of $r_c(x)$ as follows,
\begin{align} \nonumber
\prod_{k=0}^{j}(qx+\omega-q^{k}c-\omega[k]_q)
&=q^{j+1}\prod_{k=0}^{j}(x-q^{k-1}c-\omega[k-1]_q)\\  \label{eqB}
&=q^{j+1}\frac{x-q^{-1}c-\omega[-1]_q}{x-q^{j}c-\omega[j]_q}r_c(x).
\end{align}
So, using (\ref{eqA}) and (\ref{eqB}), the relation (\ref{need-lemma-2}) can be rewritten as
\begin{eqnarray*}
&&\frac{[j+1]_q}{\left(x-q^{j}c-\omega[j]_q\right)}r_c(x)Q_n(x)
+q^{j+1}\frac{x-q^{-1}c-\omega[-1]_q}{x-q^{j}c-\omega[j]_q}r_c(x)\mathscr{D}_{q,\omega}Q_n(x)\\
&=&P_n(x)\mathscr{D}_{q,\omega}f_{1,c,n}(x)+f_{1,c,n}(qx+\omega)\mathscr{D}_{q,\omega}P_n(x)\\
&+&P_{n-1}(x)\mathscr{D}_{q,\omega}g_{1,c,n}(x)+g_{1,c,n}(qx+\omega)\mathscr{D}_{q,\omega}P_{n-1}(x).
\end{eqnarray*}
Then,
\begin{align*}
&r_c(x)\mathscr{D}_{q,\omega}Q_n(x)\\
&=\frac{x-q^{j}c-\omega[j]_q}{q^{j+1}\left(x-q^{-1}c-\omega[-1]_q\right)}P_n(x)\mathscr{D}_{q,\omega}f_{1,c,n}(x) \\
&+\frac{x-q^{j}c-\omega[j]_q}{q^{j+1}\left(x-q^{-1}c-\omega[-1]_q\right)}f_{1,c,n}(qx+\omega)\mathscr{D}_{q,\omega}P_n(x)\\
&+\frac{x-q^{j}c-\omega[j]_q}{q^{j+1}\left(x-q^{-1}c-\omega[-1]_q\right)}P_{n-1}(x)\mathscr{D}_{q,\omega}g_{1,c,n}(x)\\
&+\frac{x-q^{j}c-\omega[j]_q}{q^{j+1}\left(x-q^{-1}c-\omega[-1]_q\right)}g_{1,c,n}(qx+\omega)\mathscr{D}_{q,\omega}P_{n-1}(x)\\
&-\frac{x-q^{j}c-\omega[j]_q}{q^{j+1}\left(x-q^{-1}c-\omega[-1]_q\right)}
\frac{[j+1]_q}{\left(x-q^{j}c-\omega[j]_q\right)}\left(f_{1,c,n}(x)P_n(x)+g_{1,c,n}(x)P_{n-1}(x)\right).
\end{align*}

Now, from (\ref{ttrr}) we can observe
\begin{equation}\label{ttrr-n-2}
P_{n-2}(x)=\frac{(x-\alpha_{n-1})P_{n-1}(x)-P_n(x)}{\beta_{n-1}}, \quad n\geq 2.
\end{equation}

Therefore, using (\ref{ttrr-n-2}) and doing some algebraic manipulations, we obtain the result.
\end{proof}

\begin{lemma}\label{h-lemma-3}
Let  $\{Q_n(x)\}_{n\geq0}$ and $\{P_n(x)\}_{n\geq0}$ be the sequences of monic  orthogonal polynomials  with respect to (\ref{inner-product}) and (\ref{standart-product}), respectively. Then,
\begin{equation}\label{relation-3}
r_c(x) Q_{n-1}(x)=f_{3,c,n}(x)P_n(x)+g_{3,c,n}(x)P_{n-1}(x),\qquad n\geq1,
\end{equation}
where
\begin{eqnarray*}\label{f3}
f_{3,c,n}(x)&=&-\frac{g_{1,c,n-1}(x)}{\beta_{n-1}} ,\\ \label{g3}
g_{3,c,n}(x)&=&f_{1,c,n-1}(x)+\frac{(x-\alpha_{n-1})g_{1,c,n-1}(x)}{\beta_{n-1}}.
\end{eqnarray*}
\end{lemma}

\begin{proof}
It is enough to use Lemma  \ref{h-lemma-1} and  (\ref{ttrr-n-2}) to obtain
\begin{align*}
r_c(x)Q_{n-1}(x)
=&f_{1,c,n-1}(x)P_{n-1}(x)+g_{1,c,n-1}(x)P_{n-2}(x)\\
=&-\frac{g_{1,c,n-1}(x)}{\beta_{n-1}}P_n(x)+\left(f_{1,c,n-1}(x)+\frac{(x-\alpha_{n-1})g_{1,c,n-1}(x)}{\beta_{n-1}}\right)P_{n-1}(x),
\end{align*}
where $f_{1,c,n}(x)$ and $g_{1,c,n}(x)$ are defined in   (\ref{f1}) and (\ref{g1}), respectively.
\end{proof}

\begin{lemma}\label{h-lemma-4}
Let  $\{Q_n(x)\}_{n\geq0}$ and $\{P_n(x)\}_{n\geq0}$ be the sequences of monic  orthogonal polynomials  with respect to (\ref{inner-product}) and (\ref{standart-product}), respectively. Then,
\begin{equation*}\label{relation-4}
r_c(x)\mathscr{D}_{q,\omega}Q_{n-1}(x)=f_{4,c,n}(x)P_n(x)+g_{4,c,n}(x)P_{n-1}(x),\quad n\geq2,
\end{equation*}
where
\begin{eqnarray*}\label{f4}
f_{4,c,n}(x)&=&-\frac{g_{2,c,n-1}(x)}{\beta_{n-1}} ,\\ \label{g4}
g_{4,c,n}(x)&=&f_{2,c,n-1}(x)+\frac{(x-\alpha_{n-1})g_{2,c,n-1}(x)}{\beta_{n-1}}.
\end{eqnarray*}
\end{lemma}

\begin{proof}
The proof is identical to the one  in  Lemma \ref{h-lemma-3}, but now using (\ref{relation-2})-(\ref{g2}).
\end{proof}

We define the following functions
\begin{equation} \label{def-phi}\varphi_{c,n}^{i,j}(x):=f_{i,c,n}(x)g_{j,c,n}(x)-f_{j,c,n}(x)g_{i,c,n}(x), \qquad i,j\in\{1,2,3,4\},
\end{equation}
where  $f_{i,c,n}(x)$ and $g_{i,c,n}(x),$  $i\in\{1,2,3,4\},$ are defined in Lemmas \ref{h-lemma-1}-\ref{h-lemma-4}.

\begin{lemma}\label{h-lemma-5}
For $n\geq2$, we have
\begin{align}\label{f5}
\varphi_{c,n}^{1,3}(x)P_n(x)&=r_c(x)\big(g_{3,c,n}(x)Q_n(x)-g_{1,c,n}(x)Q_{n-1}(x)\big),\\ \label{g5}
\varphi_{c,n}^{1,3}(x)P_{n-1}(x)&=r_c(x)\big(f_{1,c,n}(x)Q_{n-1}(x)-f_{3,c,n}(x)Q_n(x)\big),
\end{align}
where the function $\varphi_{c,n}^{1,3}$ is given by (\ref{def-phi}) and the other functions are defined in  previous lemmas.
\end{lemma}

\begin{proof}
The functions $f_{1,c,n}(x)$, $g_{1,c,n}(x)$, $f_{3,c,n}(x)$ and $g_{3,c,n}(x)$ are polynomials of degree at most $j+1$ in the variable $x$, then we multiply   (\ref{relation-1}) by  $g_{3,c,n}(x)$ and  (\ref{relation-3}) by $-g_{1,c,n}(x)$ obtaining
\begin{eqnarray*}
r_c(x)g_{3,c,n}(x)Q_n(x)&=&f_{1,c,n}(x)g_{3,c,n}(x)P_n(x)+g_{1,c,n}(x)g_{3,c,n}(x)P_{n-1}(x),\\
-r_c(x)g_{1,c,n}(x)Q_{n-1}(x)&=&-f_{3,c,n}(x)g_{1,c,n}(x)P_n(x)-g_{3,c,n}(x)g_{1,c,n}(x)P_{n-1}(x).
\end{eqnarray*}
Adding both expressions we deduce (\ref{f5}). Equation (\ref{g5}) is obtained in the same way using appropriate functions.
\end{proof}

\section{Ladder operators and a second--order differential--difference equation for $Q_n(x)$} \label{sect-lo}

We are ready to obtain a second--order linear differential--difference equation satisfied by the orthogonal  polynomials $Q_n(x)$ with respect to the inner product (\ref{inner-product}). The first step is to obtain the ladder operators for this family of polynomials. To do this, we will use the key Lemmas  \ref{h-lemma-1}--\ref{h-lemma-5} obtained in the previous section.

\begin{theorem}\label{theo1}
\textbf{(Ladder Operators)}
Let  $\{Q_n(x)\}_{n\ge0}$  be a sequence of monic  orthogonal polynomials  with respect to
$$(f,g)_S=\int f(x)g(x)\varrho(x)dx+M\mathscr{D}_{q,\omega}^{(j)}f(c)\mathscr{D}_{q,\omega}^{(j)}g(c).$$
The ladder operators   $\Phi_n$ and  $\widehat{\Phi}_n$ defined by
\begin{eqnarray*}
\Phi_n&:=&\varphi_{c,n}^{3,2}(x)+\varphi_{c,n}^{1,3}(x)\mathscr{D}_{q,\omega},\qquad n\geq2,\\
\widehat{\Phi}_n&:=&\varphi_{c,n}^{1,4}(x)-\varphi_{c,n}^{1,3}(x)\mathscr{D}_{q,\omega},\qquad n\geq2,
\end{eqnarray*}
satisfy
\begin{eqnarray}\label{h-lowering}
\Phi_nQ_n(x)&=&\varphi_{c,n}^{1,2}(x)Q_{n-1}(x),\\ \label{h-raising}
\widehat{\Phi}_nQ_{n-1}(x)&=&\varphi_{c,n}^{3,4}(x)Q_n(x),
\end{eqnarray}
where the functions $\varphi_{c,n}^{i,j}(x)$ are given by (\ref{def-phi}).
\end{theorem}
\begin{proof}
To prove (\ref{h-lowering}), we first multiply  (\ref{relation-2}) by $\varphi_{c,n}^{1,3}(x).$ Then, we use (\ref{f5})--(\ref{g5}) obtaining
\begin{eqnarray*}
r_c(x)\varphi_{c,n}^{1,3}(x)\mathscr{D}_{q,\omega}Q_{n}(x)
&=&f_{2,c,n}(x)r_c(x)(g_{3,c,n}(x)Q_n(x)-g_{1,c,n}(x)Q_{n-1}(x))\\
&+&g_{2,c,n}(x)r_c(x)(-f_{3,c,n}(x)Q_n(x)+f_{1,c,n}(x)Q_{n-1}(x)).
\end{eqnarray*}
Simplifying and using (\ref{def-phi}), we find
\begin{eqnarray*}
\varphi_{c,n}^{1,3}(x)\mathscr{D}_{q,\omega}Q_{n}(x)
&=& \varphi_{c,n}^{2,3}(x) Q_n(x)+\varphi_{c,n}^{1,2}(x) Q_{n-1}(x).
\end{eqnarray*}
 Taking into account  $\varphi_{c,n}^{3,2}(x)=-\varphi_{c,n}^{2,3}(x),$ we deduce (\ref{h-lowering}). The proof of (\ref{h-raising}) is completely analogous. Therefore, we have omitted it.
\end{proof}

\medskip

Finally, we  establish the following statement.
\begin{theorem}\label{theo2} The $n$-th monic orthogonal polynomials $Q_n(x)$ satisfy the following second--order linear differential--difference equation
\begin{equation*}
\sigma_{1,c,n}(x)\mathscr{D}_{q,\omega}^{(2)}Q_n(x)+\sigma_{2,c,n}(x)\mathscr{D}_{q,\omega}Q_n(x)+\sigma_{3,c,n}(x)Q_n(x)=0, \quad n\geq2,
\end{equation*}
where
\begin{eqnarray*}
\sigma_{1,c,n}(x)&=&\varphi_{c,n}^{1,3}(x)\varphi_{c,n}^{1,3}(qx+\omega)\varphi_{c,n}^{1,2}(x),\\
\sigma_{2,c,n}(x)&=&\varphi_{c,n}^{1,3}(x)\bigg(\varphi_{c,n}^{1,2}(x)\left(\varphi_{c,n}^{3,2}(qx+\omega)+\mathscr{D}_{q,\omega}
\varphi_{c,n}^{1,3}(x)\right)\\
&-&\varphi_{c,n}^{1,2}(qx+\omega)\varphi_{c,n}^{1,4}(x)
-\varphi_{c,n}^{1,3}(x)\mathscr{D}_{q,\omega}\varphi_{c,n}^{1,2}(x)\bigg),\\
\sigma_{3,c,n}(x)&=&\varphi_{c,n}^{1,2}(qx+\omega)\left(\varphi_{c,n}^{1,2}(x)\varphi_{c,n}^{3,4}(x)-\varphi_{c,n}^{1,4}(x)\varphi_{c,n}^{3,2}(x)\right)\\
&+&\varphi_{c,n}^{1,3}(x)\left(\varphi_{c,n}^{1,2}(x)\mathscr{D}_{q,\omega}\varphi_{c,n}^{3,2}(x)-\varphi_{c,n}^{3,2}(x)
\mathscr{D}_{q,\omega}\varphi_{c,n}^{1,2}(x)\right).
\end{eqnarray*}
\end{theorem}

\begin{proof}
Once we know the corresponding ladder operators given in Theorem \ref{theo1}, we proceed as follows. From (\ref{h-raising}) we have $$\widehat{\Phi}_nQ_{n-1}(x)=\varphi_{c,n}^{3,4}(x)Q_n(x),$$ and using (\ref{h-lowering}) we find $$Q_{n-1}(x)=\frac{\Phi_nQ_n(x)}{\varphi_{c,n}^{1,2}(x)},$$ so, we obtain
\begin{equation} \label{sect4-f1}
\widehat{\Phi}_n\frac{\Phi_nQ_n(x)}{\varphi_{c,n}^{1,2}(x)}=\varphi_{c,n}^{3,4}(x)Q_n(x).
\end{equation}
Now, we analyze the left--hand side in the above expression getting
\begin{align*}
\widehat{\Phi}_n\frac{\Phi_nQ_n(x)}{\varphi_{c,n}^{1,2}(x)}
&=\frac{\varphi_{c,n}^{1,4}(x)}{\varphi_{c,n}^{1,2}(x)}\Phi_nQ_n(x)-\varphi_{c,n}^{1,3}(x)\mathscr{D}_{q,\omega}
\frac{\Phi_nQ_n(x)}{\varphi_{c,n}^{1,2}(x)}\\
&=\frac{\varphi_{c,n}^{1,4}(x)}{\varphi_{c,n}^{1,2}(x)}\left(\varphi_{c,n}^{3,2}(x)Q_n(x)+\varphi_{c,n}^{1,3}(x)\mathscr{D}_{q,\omega}Q_n(x)\right)\\
&\quad-\varphi_{c,n}^{1,3}(x)\left( \mathscr{D}_{q,\omega}\frac{\varphi_{c,n}^{3,2}(x)Q_n(x)}{\varphi_{c,n}^{1,2}(x)}
+\mathscr{D}_{q,\omega}\frac{\varphi_{c,n}^{1,3}(x)\mathscr{D}_{q,\omega}Q_n(x)}{\varphi_{c,n}^{1,2}(x)}\right).
\end{align*}
We use (\ref{h-rule-product}-\ref{h-quotient-rule}) to compute $\displaystyle \mathscr{D}_{q,\omega}\frac{\varphi_{c,n}^{3,2}(x)Q_n(x)}{\varphi_{c,n}^{1,2}(x)}$ and $\displaystyle \mathscr{D}_{q,\omega}\frac{\varphi_{c,n}^{1,3}(x)\mathscr{D}_{q,\omega}Q_n(x)}{\varphi_{c,n}^{1,2}(x)}$,
then substituting them  into (\ref{sect4-f1}), we deduce
\begin{eqnarray*}
\varphi_{c,n}^{3,4}(x)Q_n(x)&=&\frac{\varphi_{c,n}^{1,4}(x)}{\varphi_{c,n}^{1,2}(x)}\left(\varphi_{c,n}^{3,2}(x)Q_n(x)+\varphi_{c,n}^{1,3}(x)\mathscr{D}_{q,\omega}Q_n(x)\right)\\
&-&\varphi_{c,n}^{1,3}(x)
\left(
\frac{\varphi_{c,n}^{1,2}(x)\mathscr{D}_{q,\omega}\varphi_{c,n}^{3,2}(x)-\varphi_{c,n}^{3,2}(x)\mathscr{D}_{q,\omega}\varphi_{c,n}^{1,2}(x)}
{\varphi_{c,n}^{1,2}(x)\varphi_{c,n}^{1,2}(qx+\omega)}Q_n(x)\right.\\
&+&\frac{\varphi_{c,n}^{3,2}(qx+\omega)}{\varphi_{c,n}^{1,2}(qx+\omega)}\mathscr{D}_{q,\omega}Q_n(x)\\
&+&\frac{\varphi_{c,n}^{1,2}(x)\mathscr{D}_{q,\omega}\varphi_{c,n}^{1,3}(x)-\varphi_{c,n}^{1,3}(x)\mathscr{D}_{q,\omega}
\varphi_{c,n}^{1,2}(x)}
{\varphi_{c,n}^{1,2}(x)\varphi_{c,n}^{1,2}(qx+\omega)}\mathscr{D}_{q,\omega}Q_n(x)\\
&+&\left.\frac{\varphi_{c,n}^{1,3}(qx+\omega)}{\varphi_{c,n}^{1,2}(qx+\omega)}\mathscr{D}_{q,\omega}^{(2)}Q_n(x)
\right).
\end{eqnarray*}
Finally, to deduce the second--order differential--difference equation for the nonstandard polynomials $Q_n,$ it only remains to multiply the previous expression by $\varphi_{c,n}^{1,2}(x)\varphi_{c,n}^{1,2}(qx+\omega)$ and simplify.
\end{proof}

\bigskip

Summarizing, we have obtained ladder operators for a wide family of Sobolev--type orthogonal polynomials with respect to an inner product involving the Hahn discrete operator. Furthermore, we have proved that these families of nonstandard orthogonal polynomials satisfy a second--order differential--difference equation whose polynomial coefficients can be computed explicitly. For this reason, we have included three appendices with all the details corresponding to the operators $\mathscr{D}_{q}$, $\Delta$ and $\frac{d}{dx}.$ The formulas given in these appendices can be implemented in a computer program which will be useful in practice. We will discuss this in a forthcoming paper.

\section*{Disclosure statement}

The authors declare that they have no competing interests.

\section*{Authors' contributions}

The authors have contributed equally to the writing of this paper.

\section*{Acknowledgments}

We thank the referees for their constructive suggestions and remarks. Undoubtedly, their comments have improved this paper.

\section*{Funding}

\noindent The authors  are partially supported by the Ministry of Science, Innovation, and
Universities of Spain and the European Regional Development Fund (ERDF), grant MTM2017-89941-P; they are
also partially supported by ERDF and Consejer\'{\i}a de Econom\'{\i}a, Conocimiento, Empresas y Universidad de la
Junta de Andaluc\'{\i}a (grant UAL18-FQM-B025-A) and by Research Group FQM-0229 (belonging to Campus of
International Excellence CEIMAR). The author J.J.M-B. is also partially supported by the research centre CDTIME
of Universidad de Almer\'{\i}a and by Junta de Andaluc\'{\i}a and ERDF, Ref. SOMM17/6105/UGR.
The author G.F. acknowledges the support of the National Science Center (Poland) via grant OPUS2017/25/B/BST1/00931.

\newpage
\section*{Appendix A. The $\mathscr{D}_{q}$ difference operator}

 The $\mathscr{D}_{q}$ difference operator is obtained by taking $\omega = 0$ in the expression (\ref{hahn-operator}) of the Hahn difference operator $\mathscr{D}_{q,\omega}.$  We provide explicit expressions for all the functions involved in the construction of the ladder operators (Theorem \ref{theo1}) for the discrete Sobolev orthogonal polynomials $Q_n$ as well as for the coefficients of the second--order difference equation satisfied by $Q_n.$  All these expressions can be computed using only the standard polynomials $P_n, $ and their properties such as relation (\ref{Low-rela-Pn}). We proceed in a similar way in the following two appendices.
{\small
\begin{align*}
\rho_{n,j,c}&=\frac{\mathscr{D}_{q}^{(j)}P_n(c)}{1+MK_{n-1}^{(j,j)}(c,c)},\\
r_c(x)&=\prod_{k=0}^{j}(x-q^kc),\\
f_{1,c,n}(x)&=r_c(x)-\frac{M\rho_{n,j,c}}{h_{n-1}}
\left(\sum_{i=0}^j\frac{(q;q)_j\mathscr{D}_{q}^{(i)}P_{n-1}\left(c\right)\prod_{k=0}^{i-1}(x-q^kc)}
{(q;q)_{i}(1-q)^{j-i}}\right),\\
g_{1,c,n}(x)&=\frac{M\rho_{n,j,c}}{h_{n-1}}\left(\sum_{i=0}^j\frac{(q;q)_j\mathscr{D}_{q}^{(i)}
P_{n}\left(c\right)\prod_{k=0}^{i-1}(x-q^kc)} {(q;q)_{i}(1-q)^{j-i}}\right),\\
f_{2,c,n}(x)&=\frac{x-q^{j}c}{q^{j+1}\left(x-q^{-1}c\right)}\left(\mathscr{D}_{q}f_{1,c,n}(x)+f_{1,c,n}(qx)\frac{B_n(x)}{A(x)}
-g_{1,c,n}(qx)\frac{C_{n-1}(x)}{\beta_{n-1}A(x)}\right) \\
&  \quad-\frac{[j+1]_q}{q^{j+1}\left(x-q^{-1}c\right)}f_{1,c,n}(x),\\
g_{2,c,n}(x)&=\frac{x-q^{j}c}{q^{j+1}\left(x-q^{-1}c\right)}\left(\mathscr{D}_{q}g_{1,c,n}(x)+f_{1,c,n}(qx)\frac{C_n(x)}{A(x)}\right.\\
&\quad\left.+g_{1,c,n}(qx)\left(\frac{B_{n-1}(x)}{A(x)}+\frac{(x-\alpha_{n-1})C_{n-1}(x)}{A(x)\beta_{n-1}}\right)
-\frac{[j+1]_q}{\left(x-q^{j}c\right)}g_{1,c,n}(x)\right),\\
f_{3,c,n}(x)&=-\frac{g_{1,c,n-1}(x)}{\beta_{n-1}} ,\\
g_{3,c,n}(x)&=f_{1,c,n-1}(x)+\frac{(x-\alpha_{n-1})g_{1,c,n-1}(x)}{\beta_{n-1}},\\
f_{4,c,n}(x)&=-\frac{g_{2,c,n-1}(x)}{\beta_{n-1}},\\
g_{4,c,n}(x)&=f_{2,c,n-1}(x)+\frac{(x-\alpha_{n-1})g_{2,c,n-1}(x)}{\beta_{n-1}},\\
\sigma_{1,c,n}(x)&=\varphi_{c,n}^{1,3}(x)\varphi_{c,n}^{1,3}(qx)\varphi_{c,n}^{1,2}(x),\\
\sigma_{2,c,n}(x)&=\varphi_{c,n}^{1,3}(x)\bigg(\varphi_{c,n}^{1,2}(x)\big(\varphi_{c,n}^{3,2}(qx)+\mathscr{D}_{q}\varphi_{c,n}^{1,3}(x)\big)
-\varphi_{c,n}^{1,2}(qx)\varphi_{c,n}^{1,4}(x)-\varphi_{c,n}^{1,3}(x)\mathscr{D}_{q}\varphi_{c,n}^{1,2}(x)\bigg),\\
\sigma_{3,c,n}(x)&=\varphi_{c,n}^{1,2}(qx)\big(\varphi_{c,n}^{1,2}(x)\varphi_{c,n}^{3,4}(x)-\varphi_{c,n}^{1,4}(x)\varphi_{c,n}^{3,2}(x)\big)\\
&\quad+\varphi_{c,n}^{1,3}(x)\left(\varphi_{c,n}^{1,2}(x)\mathscr{D}_{q}\varphi_{c,n}^{3,2}(x)-\varphi_{c,n}^{3,2}(x)
\mathscr{D}_{q}\varphi_{c,n}^{1,2}(x)\right),
\end{align*}
where $$\varphi_{c,n}^{i,j}(x)=f_{i,c,n}(x)g_{j,c,n}(x)-f_{j,c,n}(x)g_{i,c,n}(x), \qquad i,j\in\{1,2,3,4\}.
$$
}

\newpage
\section*{Appendix B. The forward difference operator}

 The forward difference operator, $\Delta$,  is obtained by taking $ \omega =1 $ and $q=1$ in the expression (\ref{hahn-operator}) of the Hahn difference operator $\mathscr{D}_{q,\omega}.$
Taking into account  (see \cite{Koekoek-book-hyper})
\begin{equation*}
\lim_{q\to1}[a]_q=a,\qquad \qquad
\lim_{q\to1}\frac{(q^a;q)_n}{(1-q)^a}=(a)_n,
\end{equation*}
 we obtain,

{\small
\begin{align*}
\rho_{n,j,c}&=\frac{\Delta^{(j)}P_n(c)}{1+MK_{n-1}^{(j,j)}(c,c)},\\
r_c(x)&=\prod_{k=0}^{j}(x-c-k)=(x-c-j)_{j+1},\\
f_{1,c,n}(x)&=r_c(x)-\frac{M\rho_{n,j,c}}{h_{n-1}}\left(\sum_{i=0}^{j} \frac{j!\Delta^{i}P_{n-1}(c)\prod_{k=0}^{i-1}(x-c-k)}{i!}\right),\\
g_{1,c,n}(x)&=\frac{M\rho_{n,j,c}}{h_{n-1}}\left(\sum_{i=0}^{j} \frac{j!\Delta^{i}P_{n}(c)\prod_{k=0}^{i-1}(x-c-k)}{i!}\right),\\
f_{2,c,n}(x)&=\frac{x-c-j}{x-c+1}\left(\Delta f_{1,c,n}(x)-g_{1,c,n}(x+1)\frac{C_{n-1}(x)}{A(x)\beta_{n-1}}+
\frac{B_{n}(x)}{A(x)}f_{1,c,n}(x+1)\right)\\
&\quad-\frac{j+1}{x-c+1}f_{1,c,n}(x),\\
g_{2,c,n}(x)&=\frac{x-c-j}{x-c+1}\left(\Delta g_{1,c,n}(x)+f_{1,c,n}(x+1)\frac{C_n(x)}{A(x)}\right.\\
&\quad\left.+g_{1,c,n}(x+1)\left(\frac{B_{n-1}(x)}{A(x)}+\frac{(x-\alpha_{n-1})C_{n-1}(x)}{A(x)\beta_{n-1}}\right)\right)-\frac{j+1}{x-c+1}g_{1,c,n}(x),\\
f_{3,c,n}(x)&=-\frac{g_{1,c,n-1}(x)}{\beta_{n-1}} ,\\
g_{3,c,n}(x)&=f_{1,c,n-1}(x)+\frac{(x-\alpha_{n-1})g_{1,c,n-1}(x)}{\beta_{n-1}},\\
f_{4,c,n}(x)&=-\frac{g_{2,c,n-1}(x)}{\beta_{n-1}},\\
g_{4,c,n}(x)&=f_{2,c,n-1}(x)+\frac{(x-\alpha_{n-1})g_{2,c,n-1}(x)}{\beta_{n-1}},\\
\sigma_{1,c,n}(x)&=\varphi_{c,n}^{1,3}(x)\varphi_{c,n}^{1,3}(x+1)\varphi_{c,n}^{1,2}(x),\\
\sigma_{2,c,n}(x)&=\varphi_{c,n}^{1,3}(x)\bigg(\varphi_{c,n}^{1,2}(x)\big(\varphi_{c,n}^{3,2}(x+1)+\Delta \varphi_{c,n}^{1,3}(x)\big)-\varphi_{c,n}^{1,2}(x+1)\varphi_{c,n}^{1,4}(x)
-\varphi_{c,n}^{1,3}(x)\Delta\varphi_{c,n}^{1,2}(x)\bigg),\\
\sigma_{3,c,n}(x)&=\varphi_{c,n}^{1,2}(x+1)\big(\varphi_{c,n}^{1,2}(x)\varphi_{c,n}^{3,4}(x)-\varphi_{c,n}^{1,4}(x)\varphi_{c,n}^{3,2}(x)\big)
+\varphi_{c,n}^{1,3}(x)\left(\varphi_{c,n}^{1,2}(x)\Delta \varphi_{c,n}^{3,2}(x) -\varphi_{c,n}^{3,2}(x)\Delta\varphi_{c,n}^{1,2}(x)\right),
\end{align*}
where $$\varphi_{c,n}^{i,j}(x)=f_{i,c,n}(x)g_{j,c,n}(x)-f_{j,c,n}(x)g_{i,c,n}(x), \qquad i,j\in\{1,2,3,4\}.
$$
}

\newpage

\section*{Appendix C. The derivative operator}

The derivative operator is obtained as the limiting case of the Hahn difference operator $\mathscr{D}_{q,\omega}$ when  $ \omega =0 $ and $q\to 1.$ Then, we find:

{\small
\begin{align*}
\rho_{n,j,c}&=\frac{P_n^{(j)}(c)}{1+MK_{n-1}^{(j,j)}(c,c)},\\
r_c(x)&=(x-c)^{j+1},\\
f_{1,c,n}(x)&=r_c(x)-\frac{M\rho_{n,j,c}}{h_{n-1}}\left(\sum_{i=0}^{j} \frac{j!P_{n-1}^{(i)}(c)(x-c)^{i}}{i!}\right),\\
g_{1,c,n}(x)&=\frac{M\rho_{n,j,c}}{h_{n-1}}\left(\sum_{i=0}^{j} \frac{j!P_{n}^{(i)}(c)(x-c)^{i}}{i!}\right),\\
f_{2,c,n}(x)&=f'_{1,c,n}(x)-g_{1,c,n}(x)\frac{C_{n-1}(x)}{A(x)\beta_{n-1}}+f_{1,c,n}(x)\left(\frac{B_{n}(x)}{A(x)}-\frac{j+1}{x-c}\right),\\
g_{2,c,n}(x)&=g'_{1,c,n}(x)+f_{1,c,n}(x)\frac{C_n(x)}{A(x)}
+g_{1,c,n}(x)\left(\frac{B_{n-1}(x)}{A(x)}+\frac{(x-\alpha_{n-1})C_{n-1}(x)}{A(x)\beta_{n-1}}-\frac{j+1}{x-c}\right),\\
f_{3,c,n}(x)&=-\frac{g_{1,c,n-1}(x)}{\beta_{n-1}} ,\\
g_{3,c,n}(x)&=f_{1,c,n-1}(x)+\frac{(x-\alpha_{n-1})g_{1,c,n-1}(x)}{\beta_{n-1}},\\
f_{4,c,n}(x)&=-\frac{g_{2,c,n-1}(x)}{\beta_{n-1}},\\
g_{4,c,n}(x)&=f_{2,c,n-1}(x)+\frac{(x-\alpha_{n-1})g_{2,c,n-1}(x)}{\beta_{n-1}},\\
\sigma_{1,c,n}(x)&=\left(\varphi_{c,n}^{1,3}(x)\right)^2\varphi_{c,n}^{1,2}(x),\\
\sigma_{2,c,n}(x)&=\varphi_{c,n}^{1,3}(x)\bigg(\varphi_{c,n}^{1,2}(x)\big(\varphi_{c,n}^{3,2}(x)+\left(\varphi_{c,n}^{1,3}\right)'(x)-\varphi_{c,n}^{1,4}(x)\big)
-\left(\varphi_{c,n}^{1,2}\right)'(x)\varphi_{c,n}^{1,3}(x)\bigg),\\
\sigma_{3,c,n}(x)&=\varphi_{c,n}^{1,2}(x)\big(\varphi_{c,n}^{1,2}(x)\varphi_{c,n}^{3,4}(x)-\varphi_{c,n}^{1,4}(x)\varphi_{c,n}^{3,2}(x)\big)
+\varphi_{c,n}^{1,3}(x)\left((\varphi_{c,n}^{3,2})'(x)\varphi_{c,n}^{1,2}(x)-\left(\varphi_{c,n}^{1,2}\right)'(x)\varphi_{c,n}^{3,2}(x)\right),
\end{align*}
where $$\varphi_{c,n}^{i,j}(x)=f_{i,c,n}(x)g_{j,c,n}(x)-f_{j,c,n}(x)g_{i,c,n}(x), \qquad i,j\in\{1,2,3,4\}.
$$
}

\end{document}